\documentclass[12pt]{amsart}
\usepackage{latexsym,amssymb,amscd}

\def\B'c{{\mathcal{B'}}}
\def\U'c{{\mathcal{U'}}}

\def\opn#1#2{\def#1{\operatorname{#2}}} 
\opn\chara{char}
\opn\length{\ell}
\opn\projdim{proj\,dim}
\opn\injdim{inj\,dim}
\opn\ini{in}
\opn\rank{rank}
\opn\depth{depth}
\opn\sdepth{sdepth}
\opn\height{ht}
\opn\embdim{emb\,dim}
\opn\codim{codim}

\opn\Tr{Tr}
\opn\bigrank{big\,rank}
\opn\superheight{superheight}\opn\lcm{lcm}
\opn\trdeg{tr\,deg}%
\opn\reg{reg}
\opn\lreg{lreg}
\opn\set{set}
\opn\supp{Supp}
\opn\shad{Shad}
\opn\div{div}
\opn\Div{Div}
\opn\cl{cl}
\opn\Cl{Cl}
\opn\Spec{Spec}
\opn\Supp{Supp}
\opn\supp{supp}
\opn\Sing{Sing}
\opn\Ass{Ass}
\opn\Min{Min}
\opn\size{size}
\opn\bigsize{bigsize}
\opn\lex{lex}
\opn\Ann{Ann}
\opn\Rad{Rad}
\opn\Soc{Soc}
\opn\Ker{Ker}
\opn\Coker{Coker}
\opn\Im{Im}
\opn\Hom{Hom}
\opn\Tor{Tor}
\opn\Ext{Ext}
\opn\End{End}
\opn\Aut{Aut}
\opn\id{id}

\opn\nat{nat}
\opn\GL{GL}
\opn\SL{SL}
\opn\mod{mod}
\opn\ord{ord}
\opn\aff{aff}
\opn\con{conv}
\opn\relint{relint}
\opn\st{st}
\opn\lk{lk}
\opn\cn{cn}
\opn\core{core}
\opn\vol{vol}
\opn\gr{gr}

\def\pot#1#2{#1[\kern-0.28ex[#2]\kern-0.28ex]}

\opn\dirlim{\underrightarrow{\lim}}
\opn\invlim{\underleftarrow{\lim}}
\let\union=\cup

\def\pnt{{\raise0.5mm\hbox{\large\bf.}}}

\let\to=\rightarrow

\def\Implies{\ifmmode\Longrightarrow \else
     \unskip${}\Longrightarrow{}$\ignorespaces\fi}
\def\implies{\ifmmode\Rightarrow \else
     \unskip${}\Rightarrow{}$\ignorespaces\fi}
\def\iff{\ifmmode\Longleftrightarrow \else
     \unskip${}\Longleftrightarrow{}$\ignorespaces\fi}

\let\:=\colon
\newtheorem{Theorem}{Theorem}[section]
\newtheorem{Lemma}[Theorem]{Lemma}
\newtheorem{Corollary}[Theorem]{Corollary}
\newtheorem{Proposition}[Theorem]{Proposition}

\newtheorem{Definition}[Theorem]{Definition}

\let\epsilon=\varepsilon
\let\phi=\varphi
\let\kappa=\varkappa
\numberwithin{equation}{section}

\title{Lexsegment ideals are sequentially Cohen-Macaulay}

\author[Muhammad Ishaq]{Muhammad Ishaq}
\address{Muhammad Ishaq, Abdus Salam School of Mathematical Sciences, GC
University, Lahore, 68-B New Muslim town Lahore, Pakistan.}
\email{ishaq$\_\,$maths@yahoo.com}
\thanks{The author would like to express his gratitude to ASSMS of GC University Lahore and  Ovidius University of Constan\c{t}a (Romania) for creating a very appropriate atmosphere for research work. This research is partially supported by HEC Pakistan.}
\begin{document}
\maketitle
\begin{abstract}
The associated primes of an arbitrary lexsegment ideal $I\subseteq S=K[x_1,\ldots,x_n]$ are determined. As
application it is shown that $S/I$ is a pretty clean module, therefore, $S/I$ is sequentially Cohen-Macaulay and
satisfies Stanley's conjecture.\\\\
\textbf{Mathematics Subject Classification (2000).} 13C15, 13A02.\\
\textbf{Keywords.} Lexsegment ideals, primary decomposition, pretty clean modules, sequentially Cohen-Macaulay ideals, Stanley depth.
\end{abstract}

\section{Introduction}

Let $S=K[x_1,\ldots, x_n]$ be the polynomial ring in $n$ variables over a field $K$. We consider the lexicographical
order on the monomials of $S$
induced by  $x_1>x_2>\ldots> x_n$. Let $d\geq 2$ be an integer and $\mathcal{M}_d$ the set of monomials of degree $d$
of $S$. For two monomials
$u,v\in\mathcal{M}_d$,
with $u\geq_{lex}v$, the set $$L(u,v)=\{w\in\mathcal{M}_d\ |\ u\geq_{lex}w\geq_{lex}v\}$$ is called a
\textit{lexsegment set}. A \textit{lexsegment
ideal} in $S$
is a monomial ideal of $S$ which is generated by a lexsegment set. Lexsegment ideals have been introduced by Hulett
and Martin \cite{HM}. Arbitrary
lexsegment ideals have been studied by A. Aramova, E. De Negri, and J. Herzog in \cite{ADH} and \cite{DH}. They
characterized all the lexsegment
ideals which have a linear resolution. In \cite{EOS} it was proved that a lexsegment ideal has a linear resolution if
and only if it has linear
quotients. In the same paper, for a lexsegment ideal $I\subseteq S,$ the dimension and the depth of $S/I$ are computed
and all the lexsegment ideals
which are Cohen-Macaulay are characterized. In \cite{BEOS}, the study of the associated prime ideals of a lexsegment
ideal is proposed. We answer to
this question in Section~\ref{asssection}. As an application, by extending a few results from \cite{HP} to the
multigraded modules over $S,$ we show in
Section~\ref{prettysection} that $S/I$ is a pretty
clean $S$-module for a lexsegment ideal $I\subseteq S$ (Theorem~\ref{pretty}). Consequently, it follows that $S/I$ is
sequentially Cohen-Macaulay (Corollary~\ref{seqCM}) and
the Stanley conjecture (\cite{S})  holds for $S/I$ (Corollary~\ref{Stanley}).\\
\indent\textbf{Acknowledgment.} I am grateful to Professor Viviana Ene for useful discussions and suggestions
during the preparation of the article.

\section{The associated primes of a lexsegment ideal}
\label{asssection}

Let $u=x_1^{a_1}\cdots x_n^{a_n}, v=x_1^{b_1}\cdots x_n^{b_n}\in S$ be two monomials of degree $d$ such that
$u\geq_{\lex}v$ and $I=(L(u,v))$
the lexsegment ideal determined by $u$ and $v.$ It is obviously that we may consider  $a_1>0$ since otherwise we
simply study our ideal in a
polynomial ring with a smaller number of variables. In addition, we exclude the trivial cases $u=v$ and
$I=(L(x_1^d,x_n^d))$. Moreover, we also notice that one may reduce to $b_1=0,$ that is $v$ is
of the form $v=x_q^{b_q}\cdots x_n^{b_n}$ with $q\geq 2$ and $b_q>0.$ Indeed, if $b_1 > 0,$ then, from the exact
sequence of multigraded $S$-modules
\begin{equation}
\label{x1dividesv}
0 \to \frac{S}{(I:x_1^{b_1})} \to \frac{S}{I} \to \frac{S}{(I,x_1^{b_1})}=\frac{S}{(x_1^{b_1})} \to 0,
\end{equation}
 we get
\[
\Ass(S/(I:x_1^{b_1}))\subseteq \Ass(S/I) \subseteq \Ass(S/(I:x_1^{b_1})) \union \{(x_1)\}.
\]
As $(x_1)\in \Ass(S/I)$ since it is a minimal prime of $I,$ we have $\Ass(S/I) = \Ass(S/(I:x_1^{b_1})) \union
\{(x_1)\}.$ Therefore, in order
to determine the associated primes of $I,$ we need to compute the associated primes of $(I:x_1^{b_1})$ which is a
lexsegment ideal generated in
degree $d-b_1$ whose right end, $v/x_{1}^{b_1},$ is no longer divisible by $x_1.$

To begin with, we  consider two important particular classes, namely, initial and final lexsegment ideals. We recall
that
a lexsegment ideal of the form $(L(x_1^d, v))$, $v\in \mathcal{M}_d,$ is called an \textit{initial lexsegment ideal}
determined by $v.$ We denote it by
$(L^i(v)).$ An ideal generated by a lexsegment set of the form $L(u, x_n^d)$ is called a \textit{final lexsegment
ideal} determined by $u\in \mathcal{M}_d.$ We
denote such an ideal by $(L^f(u)).$ We also  recall the following notations. For a monomial $w\in S,$ we denote
$\min(w)=\min\{i: x_i | w\}$, $\max(w)=\max\{i : x_i | w\},$ and $\supp(w)=\{i: x_i|w\}.$ In our study we are going to
use very often the following

\begin{Lemma}\label{supplemma}
Let $I=(L(u,v))$ be a lexsegment ideal with $x_1|u$, $x_1\nmid v$ and $v\neq x_n^d$. Then
\[
 \{(x_1,\ldots,x_j) : j\in \supp(v), j\neq n\}\subseteq\Ass(S/I).
\]
\end{Lemma}
\begin{proof}
For $j\in \supp(v)\backslash \{n\}$ let $w=(v/x_j)x_{n}^{d-b_n}$. We can conclude that $w\notin I$. Indeed, if $w\in
I$, then $w=m\cdot m'$ for some monomial $m\in L(u,v)$ and $m'$,
we get $w\geq_{lex}v m'$ which yields $x_n^{d-b_n}\geq_{lex}x_j m'$, which is impossible. For all $2\leq i\leq j$,
$x_iw=(x_iv/x_j)x_n^{d-b_n}\geq_{lex}vx_n^{d-b_n}$ and $x_1\nmid (x_iv/x_j)x_n^{d-b_n}$, we have
$x_iw=(x_iv/x_j)x_n^{d-b_n}\in I$. Since $x_1x_n^{d-1}\in I$, it follows that
$x_1w=x_1(v/x_j)x_n^{d-b_n}=(v/(x_jx_n^{b_n-1}))(x_1x_n^{d-1})\in I.$ Therefore $(x_1,\dots,x_j)\subseteq I:w$. Let us
assume
that there exists a monomial $z \in I:w$ with $z\notin (x_1,\dots,x_j)$, that is, $\supp(z)\subseteq \{j+1,\dots,n\}$
and $w z\in I$. Let $m\in L(u,v)$
such that $w z=m m'$ for some monomial $m'$. Then we get $v x_n^{d-b_n}z=x_j m m'\geq_{lex}x_j v m'$, which gives $z
x_n^{d-b_n}\geq_{lex}x_j m'$ which is contradict with $\supp(z)\subseteq\{j+1,\dots,n\}$. We thus have shown that
$I:w=(x_1,\dots,x_j)$, which implies that $(x_1,\dots,x_j)\in \Ass(S/I)$.
\end{proof}

\begin{Proposition}
\label{assinitial}
Let $v\in \mathcal{M}_d$ be a monomial and let $I=(L^i(v))$ the initial ideal determined by $v.$ Then
\[
\Ass(S/I)=\{(x_1,\ldots,x_j): j\in \supp(v)\cup \{n\}\}.
\]
\end{Proposition}

\begin{proof}
As we have observed before, we can assume that $v=x_q^{b_q}\cdots x_n^{b_n}$ with $q\geq 2$ and $b_q>0$. By
Lemma~\ref{supplemma} and \cite[Proposition 3.2]{EOS} we have
$\{(x_1,\ldots,x_j): j\in \supp(v)\cup \{n\}\}\subseteq \Ass(S/I)$.\\
\indent Let $P \in \Ass(S/I)$, $P\neq (x_1,\dots,x_n)$. By \cite[Proposition 4.2.9]{HH} we have $P=(x_1,\dots,x_j)$,
for some $1\leq j<n$. We want to
show that $j\in\supp(v)$. Let us assume $j\notin \supp(v)$. Since $P\supseteq I\supseteq (x_1^d,\dots,x_q^d)$ it
follows that $j>q$.
\indent Let $w$ be a monomial such that $w\notin I$ and $P=I:w$. We have $x_jw\in I$, hence there exists
$u'\geq_{lex}v$ such that $x_jw=u'm$, for
some monomial $m$. We have $x_j\nmid m$ since, otherwise, $w\in I$. For any $i<j$, we have $x_i\nmid m$ since,
otherwise,
$x_iu'/x_j>_{\lex}u'\geq_{\lex}v$, and  $w=\frac{x_iu'}{x_j}\cdot\frac{m}{x_i}\in I$, contradiction. Therefore, $m$ is
a monomial in
$K[x_{j+1},\dots,x_n]$. We can conclude that $\min(\supp(u'))\geq q$. If there exists $i\leq q-1$ such that $x_i|u'$,
then, for any $l$ such that $x_l|m$, we have $i<q<j<l$. Since $\min(\supp(u'))\geq q$, we have $(u'/x_j)x_l>_{lex} v$
by the definition of lexicographical order. Hence $w=(\frac{u'}{x_j}x_l)\cdot \frac{m}{x_l}\in I$,
contradiction again. That is $u'$ is of the form
\begin{equation}\label{1}
u'=x_q^{c_q}\cdots x_n^{c_n}\geq_{lex} x_q^{b_q}\cdots x_n^{b_n}
\end{equation}
If there exists $l$ such that $x_l|m$ and $u'x_l/x_j\geq_{\lex} v$, then as above, $w\in I$, a contradiction.
Therefore we must have
\begin{equation}\label{2}
u'x_l<x_jv \  \text{for all $l$ such that} \ x_l|m.
\end{equation}
Using (\ref{1}), and (\ref{2}) and $j\notin \supp(v)$ and by comparing the exponents in the monomials $u'$ and $v$, we
get $u'=x_q^{b_q}\cdots
x_{j-1}^{b_{j-1}}x_jx_{j+1}^{c_{j+1}}\dots x_n^{c_n}$, for some $c_{j+1},\ldots,c_n$, hence $$w=x_q^{b_q}\cdots
x_{j-1}^{b_{j-1}}x_{j+1}^{c_{j+1}}\cdots
x_n^{c_n}\cdot m$$ with $m\in K[x_{j+1},\ldots,x_n]$. Since $\frac{v}{\gcd(v,w)}\in I:w$, we must have
$\frac{v}{\gcd(v,w)}\in (x_1,\ldots,x_j)$, which
is impossible since $x_j\nmid v$ and $x_q^{b_q}\cdots x_{j-1}^{b_{j-1}}|\gcd(v,w)$.
\end{proof}

In the next step, we consider final lexsegment ideals.
First of all we observe that one should consider only  final lexsegment ideals
defined by a monomial $u\in \mathcal{M}_d$ such that $x_1|u.$ Indeed, otherwise, we are reduced to considering the
problem in a polynomial ring
with a smaller number of variables, namely $K[x_{\min(u)},\ldots,x_n]$.

\begin{Proposition}
\label{assfinal}
Let $u\in \mathcal{M}_d, u\neq x_1^d,$ with $x_1 | u$ and $I=(L^f(u))$ be the final lexsegment ideal defined by $u$.
Then
\[
\Ass(S/I)=\{(x_1,\ldots,x_n),(x_2,\ldots,x_n)\}.
\]
\end{Proposition}

\begin{proof}
By \cite[Proposition 3.2]{EOS}, we have $\depth(S/I)=0$, hence $(x_1,\dots,x_n)\in \Ass(S/I)$. On the other hand, for
any $P\in \Ass(S/I)$, we have $(x_2,\dots,x_n)\subseteq P$ since $I\supseteq (x_2,\dots,x_n)^d$. Since
$(x_2,\ldots,x_n)$ is obviously a minimal prime of $I,$ we have $(x_2,\ldots,x_n)\in \Ass(S/I).$ Therefore, the only
associated primes of $I$ are $\mathfrak{m}=(x_1,\ldots,x_n)$ and $(x_2,\ldots,x_n).$
\end{proof}

In order to compute the associated primes of an arbitrary lexsegment ideal, that is, one which is neither initial nor
final, we are going to
distinguish several cases, depending on the depth of $S/I.$ We recall that, by \cite[Proposition 3.2]{EOS},
$I=(L(u,v))$ has $\depth(S/I)=0$ if and only if
$x_nu\geq_{\lex}x_1v.$

\begin{Proposition}
\label{assdepth0}
Let $I=(L(u,v))$ be a lexsegment ideal which is neither initial nor final, with $x_1\nmid v,$ and such that
$\depth(S/I)=0$ Then
\[
\Ass(S/I)=\{(x_1,\ldots,x_j) : j\in \supp(v)\union\{n\}\}\union\{(x_2,\ldots,x_n)\}.
\]
\end{Proposition}

\begin{proof}
Since $u\neq x_1^d$, we have $I=(I,x_1^{a_1})\cap(I:x_1^{a_1})$. We get the following exact sequence of $S$-modules:
\begin{equation}\label{e1}
0\longrightarrow S/I\longrightarrow S/(I,x_1^{a_1})\oplus S/(I:x_1^{a_1})\longrightarrow
S/((I,x_1^{a_1})+(I:x_1^{a_1}))\longrightarrow 0
\end{equation}
We note that $(I,x_1^{a_1})+(I:x_1^{a_1})=(x_1^{a_1})+(I:x_1^{a_1})$. We first prove \\
$\Ass\Big(S/((I,x_1^{a_1})+(I:x_1^{a_1}))\Big)=\{(x_1,\ldots,x_n)\}$ and
$\Ass\Big(S/(I:x_1^{a_1})\Big)=\{(x_2,\ldots,x_n)\}$. If $a_1>1$, then
$I:x_1^{a_1}\supseteq(x_2,\dots,x_n)^{d-a_1+1},$ hence
$(I,x_1^{a_1})+(I:x_1^{a_1})$ is an $\mathfrak{m}$-primary monomial ideal, where $\mathfrak{m}=(x_1,\dots,x_n)$ and
$\Ass\Big(S/(I:x_1^{a_1})\Big)=\{(x_2,\ldots,x_n)\}$. Let $a_1=1$. Then we show that
$(I:x_1)\supseteq (x_2,\dots,x_n)^d$, which will imply again that $(I,x_1^{a_1})+(I:x_1^{a_1})=(I,x_1)+(I: x_1)$ is
$\mathfrak{m}$-primary. Since all the monomials $w$ of degree $d$ with $x_2^d \geq_{\lex}w\geq_{\lex}v$ are already
contained in $I,$ thus in $I:x_1$ as well. Hence, we only need to show that $L^f(v)\subseteq(I:x_1)$. Let us assume
that there exists a monomial $w$ of degree $d$ with $w<_{lex}v$ such that $w\notin
(I:x_1)$, then $x_1w<_{lex}x_1v\leq_{lex} x_nu.$ As $x_1|\frac{x_1w}{x_{\min(w)}}$, $x_1\nmid v$, we have
$\frac{x_1w}{x_{\min(w)}}>_{lex}v$. By $x_1\notin (I:x_1)$, we have $\frac{x_1w}{x_{\min(w)}}>_{lex}u$. Therefore,
$w\geq_{lex}\frac{x_nw}{x_{\min(w)}}>_{lex}\frac{x_nu}{x_1}\geq_{lex}v$, where the last inequality follows from the
condition $\depth(S/I)=0$. But then we get $w\geq_{lex}v$, a
contradiction. Consequently, we have shown that
$$\Ass\Big(S/((I,x_1^{a_1})+(I:x_1))\Big)=\{(x_1,\dots,x_n)\}$$ and $\Ass(S/(I:x_1))=\{(x_2,\ldots,x_n)\}$. Since
$\depth(S/I)=0$, hence
$\mathfrak{m}\in \Ass(S/I),$ by using the exact
sequence (\ref{e1}), we get
\begin{multline}\label{e5}
\Ass(S/I)=\Ass(S/(I,x_1^{a_1}))\cup \Ass(S/(I:x_1^{a_1}))=\\=\Ass(S/(I,x_1^{a_1}))\cup \{(x_2,\dots,x_n)\}
\end{multline}

Let us first take $a_1=1$. It is clear that $P\in \Ass_S(S/(I,x_1))$ if and only if $P=(x_1,P')$, where $P'\in
\Ass_{S'}(S'/(L^i(v)))$, where
$S'=K[x_2,\dots,x_n]$. By using Proposition~\ref{assinitial}, we get $\Ass(S/(I,x_1))=\{(x_1,\dots,x_j):j\in
\supp(v)\cup \{n\}\}$ and our proof is
completed in this case.

Let $a_1>1$. Then we consider the exact sequence of $S$-modules:
\begin{equation}\label{e3}
0\longrightarrow (I,x_1)/(I,x_1^{a_1})\longrightarrow S/(I,x_1^{a_1})\longrightarrow S/(I,x_1)\longrightarrow 0.
\end{equation}

Since $x_1^{d-1}(I,x_1)\subseteq (I,x_1^{a_1})$ and $(x_2,\dots,x_n)^{d-1}(I,x_1)\subseteq I\subseteq (I,x_1^{a_1})$,
it follows that
$\Ann_S((I,x_1)/(I,x_1^{a_1}))$ contains an $\mathfrak{m}$-primary ideal, thus we have
$$\Ass((I,x_1)/(I,x_1^{a_1}))=\{\mathfrak{m}\}.$$ From the exact
sequence (\ref{e3}) and using the above computation for $\Ass(S/(I,x_1))$, we obtain $\mathfrak{m}\in
\Ass(S/(I,x_1^{a_1}))$ and
$\Ass(S/(I,x_1^{a_1}))\subseteq \{(x_1,\dots,x_j):j\in \supp(v)\cup \{n\}\}$. The equality follows by
Lemma~\ref{supplemma}. Finally, by using
(\ref{e5}), we complete the proof.
\end{proof}

We now pass to the case $\depth(S/I)>0$ which is equivalent to the inequality $x_nu<_{\lex}x_1v.$ In particular, this
implies that $\deg_{x_l}(u)=1.$ Let
$u=x_1x_{l}^{a_l}\cdots x_n^{a_n}$ with $l\geq 2$ and $a_l >0$. The inequality $x_nu<_{\lex}x_1v$ is equivalent to
$x_l^{a_l}\cdots
x_n^{a_n+1}<_{\lex}x_q^{b_q}\cdots x_n^{b_n}.$ Therefore we have $l\geq q.$ For the next result we introduce the
following notation. For $2\leq j,t\leq n$ such that $2\leq j\leq t-2$, we denote
$P_{j,t}=(x_2,\dots,x_j,x_t,\dots,x_n)$.

\begin{Proposition}
\label{assdepth1}
Let $I=(L(u,v))$ be a lexsegment ideal with $x_1\nmid v$ and such that $\depth(S/I)>0.$
\begin{itemize}
 \item [(i)] Let $\depth(S/I)=1$. Then,
\begin{itemize}	
\item[(a)] for $a_l<d-1,$ we have
\[
\Ass(S/I)=\{(x_2,\ldots,x_n)\}
\union\{(x_1,\ldots,x_j) : j\in \supp(v)\setminus\{n\}\}\union
\]
\[\union\{P_{j,l}:j\in \supp(v), j\leq l-2 \}
\union\{P_{j,l+1}:j\in \supp(v),j\leq l-1 \};
\]
\item [(b)] for $a_l=d-1,$ we have\\$\Ass(S/I)=\{(x_2,\ldots,x_n)\}$
$\union \{(x_1,\ldots,x_j) : j\in \supp(v)\setminus\{n\}\}\union$\[\union\{P_{j,l}:j\in \supp(v),j\leq l-2
\}.
\]
\end{itemize}
\item [(ii)] Let $\depth(S/I)>1$. Then
\begin{itemize}	
\item[(a)] for $a_l<d-1,$ we have $\Ass(S/I)=$
\[\{(x_1,\ldots,x_j) : j\in \supp(v)\setminus\{n\}\}\union
\{P_{j,l}:j\in \supp(v) \}
\union\{P_{j,l+1}:j\in \supp(v) \};\]

\item [(b)] for $a_l=d-1,$ we have
\[\Ass(S/I)=\{(x_1,\ldots,x_j) : j\in \supp(v)\setminus\{n\}\}\union\{P_{j,l}:j\in \supp(v)\}
.\]
\end{itemize}
\end{itemize}
\end{Proposition}

\begin{proof}
Since $\depth(S/I)>0$, we have $\mathfrak{m}\notin \Ass(S/I)$ and $a_1=1$, then $(I:x_1)\subseteq (x_2,\ldots,x_n)$.
Hence, $\mathfrak{m}\notin \Ass(S/(I:x_1))$ from the exact sequence (\ref{e1}), where $a_1=1$, we get
$$\Ass(S/I)\subseteq
(\Ass(S/(I,x_1))\setminus \{\mathfrak{m}\})\union \Ass(S/(I:x_1)).$$ As in the the proof of
Proposition~\ref{assdepth0}, we have
$$\Ass(S/(I,x_1))\setminus \{\mathfrak{m}\}=\{(x_1,\dots,x_j):j\in \supp(v)\setminus \{n\}\}.$$ Let us first look at
$\Ass(S/(I:x_1))$. Note that
$(I:x_1)=J+L$ where $J$ is generated in degree $d-1$ by the final lexsegment $L^f(u/x_1)$, and $L$ is generated in
degree $d$ by the initial
lexsegment $L^i(v)\subseteq S'=K[x_2,\dots,x_n]$. Let us first consider $a_l<d-1$. Then, by
Proposition~\ref{assfinal}, the associated primes
of $J$ are
$P_1=(x_{l},\dots,x_n)$ and $P_2=(x_{l+1},\dots,x_n)$. Therefore, $J=Q_1\cap Q_2$, where $Q_1$ and $Q_2$ are primary
monomial ideals with
$\sqrt{Q_i}=P_i$, $i=1,2$. Similarly, we have $L=\bigcap\limits_{2\leq j\in \supp(v)\cup \{n\}}\!\!\!\!\!\!\!Q_j'$ for
some monomial primary ideals
$Q_j'$ such that $\sqrt{Q_j'}=(x_2,\dots,x_j)$ for all $j$. Then $$(I:x_1)=Q_1\cap
Q_2+\bigcap\limits_{j}Q_j'=(\bigcap\limits_j(Q_1+Q_j'))\bigcap
(\bigcap\limits_j(Q_2+Q_j'))$$ is a primary decomposition of $I:x_1$. Therefore, by the primary
decomposition of $I:x_1$ and $\mathfrak{m}\notin \Ass(S/(I:x_1))$, we get
\begin{multline*}
\Ass(S/(I:x_1))\subseteq \{(x_2,\dots,x_n)\}\cup \{P_{j,l}:j\in \supp(v),j\leq l-2\}\union\\ \union\{P_{j,l+1}:j\in
\supp(v),j\leq l-1\}.
\end{multline*}
If $a_l=d-1$, that is $u=x_1x_l^{d-1}$, then we get that $J=(x_l,\dots,x_n)^{d-1}$, hence it is a primary ideal. As
before, we get
$$\Ass(S/(I:x_1))\subseteq \{(x_2,\dots,x_n)\}\cup \{P_{j,l}:j\in \supp(v),j\leq l-2\}.$$

In order to prove (i), taking into account Lemma~\ref{supplemma}, we only need to show that $P_{j,l},j\leq l-2,$
$P_{j,l+1},j\leq l-1$, and $(x_2,\ldots,x_n)$ are associated primes of $I$. In each case, we are going to show that
one may find a monomial $f\notin
I$ such that $I:f=P_{j,l}$ or
$P_{j,l+1}$ or $(x_2,\ldots,x_n)$. We begin by proving that $(x_2,\ldots,x_n)$ is an associated prime of $I.$ By
\cite[Proposition 3.4]{EOS},
$\depth(S/I)=1$ if and only if $v=x_2^{d-1}x_j$ for some $2\leq j\leq n-2$ and $j\geq l-1$ or
$v\leq_{\lex}x_2^{d-1}x_{n-1}.$ If $v\leq_{\lex}
x_2^{d-1}x_n,$ then,
for $f=x_2^{d-1},$ we easily get $I:f=(x_2,\ldots,x_n)$ since all the monomials $x_2^{d}, x_2^{d-1}x_3,\ldots,
x_2^{d-1}x_{n}$ belong to $I.$ Let $v\geq_{\lex}x_2^{d-1}x_{n-1}.$ If $l=2,$ then we choose $f=x_1x_{n}^{d-2}$ and
observe that $x_1x_2x_n^{d-2},x_1x_3x_n^{d-2}\ldots,x_1x_n^{d-1}\in I,$ hence $I:f=(x_2,\ldots,x_n).$ Finally, for
$l\geq 3,$ we
take $f=x_1x_2^{d-1}x_n^{d-2}$ and get again the desired claim since
$x_1x_lx_n^{d-2},x_1x_{l+1}x_n^{d-2}\ldots,x_1x_n^{d-1},x_2^d,x_2^{d-1}x_3,\ldots,x_2^{d-1}x_{l-1}\in I.$  Therefore,
$(x_2,\ldots,x_n)\in \Ass(S/I)$ for $\depth(S/I)=1.$ Now let $j\in \supp(u)$ with $j\leq l-2$, we look for a monomial
$f\notin I$ such that $I:f=P_{j,l}, j\leq l-2$. Let us take $$f=x_1x_q^{b_q}\cdots
x_{j-1}^{b_{j-1}}x_j^{b_j-1}\cdots x_{l-2}^{b_{l-2}}x_{l-1}^d x_{l}^{a_l-1}x_{l+1}^{a_{l+1}}\cdots x_n^{a_n}.$$
As $j\in \supp(v)$ and $j\leq l-2$, we have $q\leq j\leq l-2<l$, then $f>_{lex}u$. Hence $f\notin I$.
We now show that
$I:f=P_{j,l}$. Let $s\in \{2,\dots,j,l,\dots,n\}.$ If $s\leq j$, then $x_sf=x_s(v'/x_j)m_1$, where $v'=x_q^{b_q}\cdots
x_j^{b_j}\cdots
x_{l-2}^{b_{l-2}}x_{l-1}^{d-(b_q+\cdots +b_{l-2})}\geq_{lex}v$ and $m_1$ is a monomial in $S$. Since $x_s(v'/x_j)\in
L(u,v)$, we get $x_s f\in I$. Let
$s\geq l$. Then
$x_s f=x_s(u/x_l) m_2$ for some monomial $m_2$, and since $x_s(u/x_l)\in L(u,v)$, we obtain $x_s f\in I$. We thus
showed that $P_{j,l}\subseteq
I:f$ for $j\leq l-2$.
Let us assume that $P_{j,l}\subsetneqq I:f$, hence there exists a monomial $w\in I:f$ such that $\supp(w)\subseteq
\{j+1,\dots,l-1\}$, that is
$w=x_{j+1}^{c_{j+1}}\dots x_{l-1}^{c_{l-1}}$, where $c_{j+1},\ldots,c_{l-1}\geq 0$. But $$w f=x_1x_q^{b_q}\cdots
x_j^{b_j-1}x_{j+1}^{c_{j+1}'}\cdots
x_{l-1}^{c_{l-1}'}x_l^{a_l-1}x_{l+1}^{a_{l+1}}\cdots x_n^{a_n},$$ and, with same arguments as above, $w f\notin I$.
Therefore, $I:f=P_{j,l}$.

Now, let $a_l<d-1$.
We show that $P_{j,l+1}\in \Ass(S/I)$ for $j\leq l-1$. If $u=x_1x_l^{a_l}x_{l+1}^{d-a_l-1}$, we take
$f=x_1x_q^{b_q}\cdots x_{j-1}^{b_j-1}\cdots
x_{l-1}^{b_{l-1}}x_{l}x_{l+1}^{d-a_l-2}$. If $u<_{lex}x_1 x_l^{a_l}x_{l+1}^{d-{a_l}-1}$, we take
$$f=x_1x_q^{b_q}\cdots
x_{j-1}^{b_{j-1}}x_j^{b_j-1}\cdots x_{l-2}^{b_{l-2}}x_{l-1}^{b_{l-1}}x_{l}^{d}x_{l+1}^{d-a_{l}-1}.$$ With similar
arguments as before, we show
that $I:f=P_{j,l+1}$ in each case.

 (ii). By \cite[Proposition 3.4]{EOS}, $\depth(S/I)>1$ if and only if $v=x_2^{d-1}x_j$, for some $2\leq j\leq
n-2$ and $l\geq j+2$. In this case $(x_2,\dots,x_n)\notin \Ass(S/I)$ and the conclusion follows directly from
Lemma~\ref{supplemma} and by looking at
$\Ass(S/(I:x_1))$. Indeed, we have the exact sequence $$0\longrightarrow S/(I:x_1)\longrightarrow S/I\longrightarrow
S/(I,x_1)\longrightarrow 0,$$ thus $\Ass(S/(I:x_1))\subseteq \Ass(S/I)\subseteq \Ass(S/(I:x_1))\cup \Ass(S/(I,x_1))$.
As $(x_1,\dots,x_j)\in \Ass(S/I)$ for all $j\in \supp(v),\,j\neq n,$ we only need to compute $\Ass(S/(I:x_1)).$ Note
that, in this case,
\[(I:x_1)=\left\{
  \begin{array}{ll}
   (L^f(u/x_1))+(x_2^d) , & \hbox{if} \ \ v=x_2^d, \\
    (L^f(u/x_1))+(x_2^{d-1})\cap (x_2^d,x_3,\dots,x_j), & \hbox{if} \ v=x_2^{d-1}x_j,\\ & \ 3\leq j\leq n-2,
  \end{array}
\right.
\]
If $v=x_2^d,$ we get, by using Proposition~\ref{assfinal}, $(I:x_1)=(x_2^d,Q_1)\cap (x_2^d,Q_2)$ where $Q_1,Q_2$ are
primary ideals with $\sqrt{Q_1}=(x_l,\dots,x_n)$ and $\sqrt{Q_2}=(x_{l+1},\dots,x_n)$, which implies that
$\Ass(S/(I:x_1))=\{P_{2,l},P_{2,l+1}\}.$ Finally, if $v=x_2^{d-1}x_j$, with $3\leq j\leq n-2$, we get, by using
Proposition~\ref{assfinal}, $$(I:x_1)=(x_2^{d-1},Q_1)\cap (x_2^{d-1},Q_2)\cap (x_2^d,x_3,\dots,x_j,Q_1)\cap
(x_2^d,x_3,\dots,x_j,Q_2),$$ where $Q_1,Q_2$ are primary and $\sqrt{Q_1}=(x_l,\dots,x_n)$,
$\sqrt{Q_2}=(x_{l+1},\dots,x_n)$. This yields $\Ass(S/(I:x_1))=\{P_{j,l},P_{j,l+1}:j\in \supp(v)\}.$
\end{proof}

\section{Lexsegment ideals are pretty clean}
\label{prettysection}

Pretty clean modules were defined in \cite{HP}. Since we are interested in  finitely generated multigraded modules
over $S,$ we recall the definition
of pretty cleanness in this frame.

\begin{Definition}[\cite{HP}]
Let $M$ be a finitely generated multigraded $S$-module. A multigraded prime filtration of $M$,
\[
\mathcal{F} :  \quad 0=M_0\subseteq M_1 \subseteq \cdots \subseteq M_{r-1}\subseteq M_r=M,
\]
where $M_i/M_{i-1}\cong S/P_i,$ with $P_i$ a monomial prime ideal, is called \textbf{pretty clean} if for all $i<j$,
$P_i\subseteq P_j$ implies
$i=j.$ In other words, a proper inclusion $P_i\subseteq P_j$ is possible only if $i>j.$ A multigraded $S$-module is
called \textbf{pretty clean} if it
admits a pretty clean filtration.
\end{Definition}

We denote by $\Supp(\mathcal{F})$ the set $\{P_1,\ldots,P_r\}$ of the prime ideals  which define the factor modules of
$\mathcal{F}.$
By \cite[Corollary 3.4.]{HP}, $\Supp(\mathcal{F})=\Ass(S/I).$

The following lemma gives a nice class of pretty clean multigraded $S$-modules.

\begin{Lemma}
\label{totalordered}
Let $M$ be a finitely generated multigraded $S$-module such that $\Ass(M)$ is totally ordered by inclusion. Then $M$
is pretty clean.
\end{Lemma}

The proof works as the proof of \cite[Proposition 5.1]{HP}, therefore we omit it.

Our aim in this section is to show that if $I\subseteq S$ is a lexsegment ideal, then $S/I$ is pretty clean. The claim
is obvious for initial and
final lexsegment ideals. Indeed, by applying  Proposition~\ref{assinitial}, Proposition~\ref{assfinal}, and the above
lemma, we get

\begin{Corollary}
\label{initial}
Let $I\subseteq S$ be an initial or final lexsegment ideal. Then $S/I$ is pretty clean.
\end{Corollary}

For arbitrary lexsegment ideals  we need  another preparatory result.

\begin{Lemma}
\label{exact sequence}

Let $0\to M^\prime \stackrel{f}{\rightarrow} M \stackrel{g}{\rightarrow} M^{\prime\prime} \to 0$ be an exact sequence
of finitely generated
multigraded $S$-modules and homogeneous morfisms.
We assume that $M^\prime$ has a multigraded pretty clean
filtration $\mathcal{F^\prime}$ and $M^{\prime\prime}$ has a  multigraded pretty clean filtration
$\mathcal{F}^{\prime\prime}$ such that for any $P\in
\Supp(\mathcal{F^\prime})$ and $Q\in \Supp(\mathcal{F}^{\prime\prime})$, we have $P\not\subseteq Q,$ that is either
$P\supseteq Q$ or $P$ and $Q$ are
incomparable by inclusion.  Then $M$ is pretty clean.
\end{Lemma}

\begin{proof}
Let $\mathcal{F}^\prime: 0=M_0^\prime\subseteq \cdots \subseteq M^\prime_r=M^\prime $ be the filtration of $M^\prime$
and
$\mathcal{F}^{\prime\prime}:  0=M_0^{\prime\prime}\subseteq \cdots \subseteq M^{\prime\prime}_s=M^{\prime\prime} $ the
filtration of
$M^{\prime\prime}.$ Then, by hypothesis, the following filtration,
\[
0=f(M^\prime_0)\subseteq \cdots \subseteq f(M^\prime_r)=f(M^\prime)=g^{-1}(0)\subseteq \cdots \subseteq
g^{-1}(M^{\prime\prime}_s)=M
\]
is a multigraded prime filtration of $M,$ hence $M$ is pretty clean.
\end{proof}

The first consequence that one  derives from the above lemma is that we can reduce, as in the previous section, to the
case when $v,$ the right
end of the lexsegment set which generates the lexsegment ideal, is not divisible by $x_1.$ Indeed, if
$\deg_{x_1}(v)=b_1>0,$ looking at the exact
sequence (\ref{x1dividesv}), we see that, in order to prove that $S/I$ is pretty clean, it is enough  to show that
$S/(I: x_1^{b_1})$ is pretty clean
since $\Ass(S/(I: x_1^{b_1}))$ obviously does not contain $(x_1).$

\begin{Theorem}
\label{pretty}
Let $I\subseteq S$ be a lexsegment ideal. Then $S/I$ is a pretty clean module.
\end{Theorem}

The proof of the theorem will follow from Corollary~\ref{initial} and the next two lemmas. As in the previous section,
we consider separately the
cases when $\depth(S/I)=0$ and $\depth(S/I)>0.$

\begin{Lemma}
\label{depth0pretty}
Let $I$ be a lexsegment ideal which is neither initial nor final and such that $\depth(S/I)=0$ and $x_1\nmid v.$ Then
$S/I$ is pretty clean.
\end{Lemma}

\begin{proof}
Let $u=x_1^{a_1}\dots x_n^{a_n}$ with $a_1>0$ and $x_q^{b_q}\dots x_n^{b_n}$ with $q\geq 2$ and $b_q>0$. We consider
the exact sequence of multigraded modules:
\begin{equation}\label{es}
0\longrightarrow (I:x_1^{a_1})/I\longrightarrow S/I\longrightarrow S/(I:x_1^{a_1})\longrightarrow 0.
\end{equation}
As $x_1^{a_1}\in \Ann_S((I:x_1^{a_1})/I)$, we get $x_1\in P$ for all $P\in \Ass((I:x_1^{a_1})/I)$. On the other hand,
as we already noticed in the proof of Proposition~\ref{assdepth0}, $\Ass(S/(I:x_1^{a_1}))=\{(x_2,\dots,x_n)\}.$
 By Proposition~\ref{assdepth0}, we have $$\Ass(S/I)=\{(x_1,\dots,x_j):j\in \supp(v)\cup \{n\}\}\cup
 \{(x_2,\dots,x_n)\},$$ which implies that
 $\Ass((I:x_1^{a_1})/I)=\{(x_1,\dots,x_j):j\in \supp(v) \cup \{n\}\}$, thus by Lemma~\ref{totalordered},
 $(I:x_1^{a_1})/I$ and $S/(I:x_1^{a_1})$ are pretty clean $S$-modules. Next we apply Lemma~\ref{exact sequence} and
 conclude that $S/I$ is pretty clean.
\end{proof}

\begin{Lemma}
\label{depth>0pretty}
Let $I$ be a lexsegment ideal  such that $\depth(S/I)>0$ and $x_1\nmid v.$ Then $S/I$ is pretty clean.
\end{Lemma}

\begin{proof}
As we have seen before, since $\depth(S/I)>0,$ $u$ and $v$ have the following form: $u=x_1x_l^{a_l}\cdots x_n^{a_n}$
with $l\geq 2$ and
$a_l>0$, $v=x_q^{b_q}\cdots x_n^{b_n}$ with $q\geq 2$ and $b_q>0.$ Moreover, we have $l\geq q.$ As in the first part
of the proof of Lemma~
\ref{depth0pretty}, by using the exact sequence of multigraded $S$-modules
\begin{equation}
\label{seq2}
0\to \frac{(I:x_1)}{I}\to \frac{S}{I} \to \frac{S}{(I:x_1)} \to 0,
\end{equation}
it is enough to show that $(I:x_1)/I$ and $S/(I:x_1)$ are pretty clean and no prime ideal of the pretty clean
filtration of $(I:x_1)/I$ is
strictly contained in a prime ideal of the pretty clean filtration of $S/(I:x_1)$. We first observe that since $x_1\in
\Ann_S((I:x_1)/I)$, we have
$x_1\in P$ for all the prime ideals $P\in \Ass((I:x_1)/I)$. On the other hand, since $x_1$ is regular on $S/(I:x_1)$,
it follows that
$x_1\not\in P$ for all $P\in\Ass(S/(I:x_1)).$ By Proposition~\ref{assdepth1}, we get
$\Ass((I:x_1)/I)=\{(x_1,\ldots,x_j): j\in
\supp(v)\setminus\{n\}\},$ therefore $(I:x_1)/I$ is a pretty clean module since its associated primes are totally
ordered by inclusion.

If $u=x_1x_{l}^{d-1}$, it follows, by  Proposition~\ref{assdepth1}, that
$\Ass(S/(I:x_1))\subseteq\{(x_2,\ldots,x_j,x_l,\\\ldots,x_n): j\in
\supp(v)\}\cup \{(x_2,\ldots,x_n)\},$ thus it is totally ordered by inclusion, which shows that $S/(I:x_1)$ is pretty
clean. The same argument works if $u <_{\lex}x_1x_l^{d-1}$ and $q=l.$ In both
cases, it is clear that for all $P\in \Ass((I:x_1)/I)$ and $P^\prime\in \Ass(S/(I:x_1))$ we have $P\not\subseteq
P^\prime.$ We then may conclude that in
these cases $S/I$ is a pretty clean module.

It remains to consider $\deg_{x_l}(v)<d-1$ and $q\leq l-1.$ We are going to show that $S/(I:x_1)$ is pretty clean
which will end our proof.
Note that one may decompose $(I:x_1)$ as $(I:x_1)=J+L$ where $J$ is generated in degree $d-1$ by the final lexsegment
$L^f(u/x_1)\subseteq
K[x_l,\ldots,x_n]$, and $L$ is generated in degree $d$ by the initial lexsegment $L^i(v)\subseteq K[x_2,\ldots,x_n].$
Let
$(L^i(v))=\bigcap_{j\in \supp(v)\cup\{n\}} Q_j$ be the irredundant primary decomposition of $(L^i(v))$ where  $Q_j$
are monomial primary ideals with
$\sqrt{Q_j}=(x_2,\ldots,x_j), j\in
\supp(v)\cup\{n\}$. Let $M=(I:x_1):x_l^d=(J+(L^i(v))):x_l^d=
J:x_l^d + (L^i(v)):x_l^d.$  It is easily seen that $J:x_l^d $ is a monomial $(x_{l+1},\ldots,x_n)$-primary ideal. In
addition, we have
$(L^i(v)):x_l^d=(\bigcap_{j\in \supp(v)\cup\{n\}} Q_j): x_l^d=\bigcap_{j\in \supp(v)\cup\{n\}} (Q_j: x_l^d)=
(\bigcap_{j\in \supp(v)\atop j\leq l-1} (Q_j: x_l^d))\bigcap (\bigcap_{j\in \supp(v)\cup\{n\}\atop j\geq l} (Q_j:
x_l^d)).$ In the last intersection,
each of the primary monomial ideals contains a power of $x_l,$ therefore $Q_j: x_l^d=S$ for all $j\geq l.$ It follows
that
$(L^i(v)):x_l^d=\bigcap_{j\in \supp(v)\atop j\leq l-1} (Q_j: x_l^d).$ This implies that $M=\bigcap_{j\in \supp(v)\atop
j\leq l-1}
(J:x_l^d + Q_j: x_l^d)=\bigcap_{j\in \supp(v)\atop j\leq l-1} (J:x_l^d + Q_j)$ is an iredundant primary decomposition
of $M$ which
gives $\Ass(S/M)=\{(x_2,\ldots,x_j,x_{l+1},\ldots,x_n): j\in \supp(v), j\leq l-1\}$. It is clear that $M\supseteq
I:x_1,$ hence we have the exact
sequence of multigraded $S$-modules
\[
0 \to \frac{M}{(I:x_1)}\to \frac{S}{(I:x_1)}\to \frac{S}{M} \to 0.
\]
On the other hand, it is also clear that $x_l^{d}M\in I:x_1$, which implies that $x_l^d\in \Ann(M/(I:x_1)).$ In
particular, it follows that
$x_l\in P$ for all $P\in \Ass(M/(I:x_1))$. From the above sequence and by using the form of $\Ass(S/(I:x_1))$ we
finally get
$\Ass(M/(I:x_1))=\{(x_2,\ldots, x_j, x_l,\ldots,x_n): j\in \supp(v)\}$, hence $M/(I:x_1)$ is pretty clean. Moreover,
there is no proper
inclusion of the type $P\subseteq P^\prime$ where $P\in \Ass(M/(I:x_1))$ and $P^\prime\in \Ass(S/M),$ hence, by
Lemma~\ref{exact sequence}, $S/(I:x_1)$
is pretty clean.
\end{proof}

Theorem~\ref{pretty}  and Corollary 4.3. in \cite{HP} yield the following

\begin{Corollary}
\label{seqCM}
Let $I\subseteq S$ be a lexsegment ideal. Then $S/I$ is sequentially Cohen-Macaulay.
\end{Corollary}

Moreover, from Theorem \ref{pretty} and  \cite[Theorem 6.5.]{HP} we get the following

\begin{Corollary}
\label{Stanley}
Let $I\subseteq S$ be a lexsegment ideal. Then $S/I$ satisfies the Stanley conjecture, that is we have the
inequality
$\sdepth(S/I)\geq \depth(S/I),$ where $\sdepth(S/I)$ is the Stanley depth of $S/I$.
\end{Corollary}


\begin{thebibliography}{99}

\bibitem  {ADH} A. Aramova, E. De Negri, J. Herzog, \textit{ Lexsegment ideals with linear resolutions,} Illinois J.
    Math., {\bf 42}(3) (1998), 509--523.

\bibitem {BEOS} V. Bonanzinga, V. Ene, A .Olteanu, L. Sorrenti, \textit{ An overview on the minimal free resolutions
    of lexsegment
ideals}, in Combinatorial Aspects of Commutative Algebra, V. Ene, E. Miller, Eds, Contemporary Mathematics, AMS, {\bf
502}, 2009, 5--24.
\bibitem  {DH} E. De Negri, J. Herzog, \textit{ Completely lexsegment ideals,} Proc. Amer. Math. Soc., {\bf 126}(12),
    1998, 3467--3473.


\bibitem {EOS} V. Ene, A.Olteanu, L. Sorrenti, \textit{ Properties of lexsegment ideals}, Osaka J. Math., {\bf 47},
    2010, 1--21.

\bibitem {HM} H. Hulett, H.M. Martin, \textit{  Betti numbers of lex-segment ideals,} J. Algebra, {\bf 275}, 2004,
    629--638.
\bibitem {HH}J. Herzog, T. Hibi, \textit{Monomial Ideals}, Graduate Texts in Mathematics {\bf 260}, Springer, 2010.
\bibitem {HP} J. Herzog, D. Popescu, \textit{ Finite filtrations of modules and shellable multicomplexes}, Manuscripta
    Math., {\bf 121}, 2006, 385--410.

\bibitem {S} R. Stanley, \textit{Linear Diophantine equations and local cohomology}, Inventiones Mathematicae {\bf
    68}, 1982,
175–-193.

\end{thebibliography}
\end{document}